\documentclass[11pt,reqno]{amsart}

\usepackage{amsmath,amssymb,mathtools}
\usepackage{enumitem}
\usepackage{microtype}
\usepackage[colorlinks=true,linkcolor=blue,citecolor=red,urlcolor=blue]{hyperref}

\numberwithin{equation}{section}

\newtheorem{theorem}{Theorem}[section]
\newtheorem{proposition}[theorem]{Proposition}
\newtheorem{lemma}[theorem]{Lemma}

\theoremstyle{definition}
\newtheorem{definition}[theorem]{Definition}
\theoremstyle{remark}

\newcommand{\R}{\mathbb{R}}
\newcommand{\Sph}{\mathbb{S}}
\newcommand{\inner}[2]{\left\langle #1,#2\right\rangle}
\newcommand{\rec}{\operatorname{rec}}
\newcommand{\Vol}{\operatorname{Vol}}
\newcommand{\Int}{\operatorname{int}}
\newcommand{\id}{\operatorname{Id}}
\newcommand{\cH}{\mathcal{H}}
\newcommand{\cR}{\mathcal{R}}

\title[Rigidity of Immersed Gauss Curvature Shrinkers]
{Rigidity of Immersed Gauss Curvature Shrinkers in Euclidean Space}
\author[J. Wang]{Jiahao Wang}
\date{}

\subjclass[2020]{Primary 53C42; Secondary 53C24, 52A20}

\keywords{Gauss curvature shrinker, immersed hypersurface, rigidity}

\address{Department of Mathematical Sciences, Tsinghua University, Beijing
100084, P.R. China}

\email{jh-wang24@mails.tsinghua.edu.cn}

\begin{document}

\begin{abstract}
For every $m\geq2$, we prove that any connected closed oriented smooth
immersion $X:M^m\to\R^{m+1}$ satisfying
\[
    K=\inner{X}{\nu}
\]
is an embedding onto the unit sphere.  No convexity is assumed.  
At a point maximizing $|X|$, we find a strictly convex component of
$\{K>0\}$.  If this component does not coincide with $M$, its polar immersion bounds
an unbounded convex set.  A finite area lemma then extends the local
surface area formula for the polar body to a global measure identity.  Using
the vanishing first moment of surface area measure, we show that the origin
lies in the interior of the polar body, contradicting the existence of
recession directions.  Thus the hypersurface is strictly convex, and the
classification of strictly convex shrinkers completes the proof.
\end{abstract}

\maketitle

\section{Introduction}

Let $X(\cdot,t):M^m\to\R^{m+1}$ be a family of smooth closed strictly convex
hypersurfaces with outward unit normal $\nu$ and Gauss--Kronecker
curvature $K$.  The $\alpha$-Gauss curvature flow is
\begin{equation}\label{eq:GCF-intro}
    \partial_tX=-K^\alpha\nu,
    \qquad \alpha>0.
\end{equation}

The case $\alpha=1$ was introduced by Firey as a model for the wearing of
stones.  Firey~\cite{Firey1974} proved asymptotic roundness for centrally
symmetric convex surfaces, subject to existence and regularity, and
conjectured that the symmetry assumption could be removed.
Tso~\cite{Tso1985} established the basic existence and regularity theory
for compact smooth strictly convex hypersurfaces, and
Andrews~\cite{Andrews1999} proved Firey's conjecture in $\R^3$.

A homothetically shrinking solution keeps the same shape while its scale
changes.  After placing the shrinking point at the origin and fixing the
scale, the resulting time-independent profile $X$ satisfies
\begin{equation}\label{eq:alpha-shrinker-intro}
    K^\alpha=\inner{X}{\nu}.
\end{equation}
We call such a profile an $\alpha$-Gauss shrinker.  The convergence results
of Andrews~\cite{Andrews2000}, Guan and Ni~\cite{GuanNi2017}, and Andrews,
Guan, and Ni~\cite{AndrewsGuanNi2016} show, in the relevant ranges of
$\alpha$, that a rescaled strictly convex flow converges to a smooth
strictly convex shrinker.  Thus, once this convergence is known, identifying
the limiting shape reduces to classifying the shrinkers.

Several important classification results are known.  At the affine
critical exponent $\alpha=1/(m+2)$, Calabi~\cite{Calabi1972} showed that
the strictly convex closed shrinkers are ellipsoids.  Choi and
Daskalopoulos~\cite{ChoiDaskalopoulos2016} classified strictly convex
closed shrinkers for $1/m<\alpha<1+1/m$.  Brendle, Choi, and
Daskalopoulos~\cite{BrendleChoiDaskalopoulos2017} completed the range
$\alpha\geq1/(m+2)$: every closed smooth strictly convex solution of
\eqref{eq:alpha-shrinker-intro} is a round sphere when
$\alpha>1/(m+2)$, while at $\alpha=1/(m+2)$ it is an ellipsoid.  In
particular, for the classical Gauss curvature flow $\alpha=1$, the only
strictly convex closed shrinker is the round sphere.  Ivaki and
Milman~\cite{IvakiMilman2023} later gave a convex geometric proof of this
isotropic uniqueness result and obtained related extensions.

The Brendle--Choi--Daskalopoulos argument relies on strict convexity.  
It makes the second fundamental form positive definite, so
its inverse $b=A^{-1}$ exists and the operator
$\alpha K^\alpha b^{ij}\nabla_i\nabla_j$ is elliptic.  Their test functions
and maximum principle argument use both $b$ and the reciprocal of the
smallest principal curvature.  At a point where the smallest principal curvature 
vanishes, these quantities are not defined and the operator degenerates, while 
at a nonconvex point, the operator need not be elliptic.  The study of weakly convex 
flows with flat sides shows that this is a real issue. Daskalopoulos and
Hamilton~\cite{DaskalopoulosHamilton1999} proved short-time existence and
regularity up to the boundary between the flat and curved parts, and
Daskalopoulos and Lee~\cite{DaskalopoulosLee2004} studied that boundary for
as long as the flat part remains.  These results begin with weak convexity
and do not imply convexity for an arbitrary immersed stationary solution.

To pass from a strictly convex component to global strict convexity, we use
the cone-volume measure of a convex body.
The logarithmic Minkowski problem asks which measures on the sphere arise
as cone-volume measures of convex bodies.  For a convex body
$C\subset\R^{m+1}$ with the origin in its interior, its cone-volume measure
is $\frac{1}{m+1}h_C\,dS_C$, where $h_C$ is the support function and $S_C$
is the surface area measure, see \cite{BoroczkyLutwakYangZhang2013}.
In our proof, the product $h_C\,dS_C$ arises for a convex body that is
initially known only to contain the origin.

The present argument does not require an existence or uniqueness theorem for the
logarithmic Minkowski problem.  It uses only an interiority property of the measure
equation
\[
    h_C\,dS_C=d\omega.
\]
Saroglou~\cite[Remark~3.5]{Saroglou2021} observed that if a convex body
containing the origin satisfies this identity, then the origin lies in its
interior. We use this conclusion to show that the strictly convex component 
constructed in Section~\ref{sec:convex-component-polar-body} must equal $M$.

The standard strictly convex shrinker $K=\langle X,\nu\rangle$ is therefore completely
understood.  The remaining question is whether convexity follows from the
shrinker equation itself.  We answer this affirmatively for every closed oriented
hypersurface, without assuming any convexity. We use the convention $A=d\nu$ for the 
shape operator and write $K=\det A$ for the Gauss curvature. Our main result is as follows.

\begin{theorem}[Main theorem]\label{thm:main}
Let $X:M^m\longrightarrow\R^{m+1}$ be a smooth immersion of a connected closed orientable hypersurface
in $\R^{m+1}$ with $m\geq2$.  Suppose that a globally defined smooth unit normal
$\nu$ satisfies
\begin{equation}\label{eq:shrinker}
    K=\inner{X}{\nu}
    \qquad\text{on }M.
\end{equation}
Then $X$ is an embedding and
\[
    X(M)=\Sph^m(1):=\{x\in\R^{m+1}:|x|=1\}.
\]
\end{theorem}

We briefly describe the proof.  At a maximum point of $|X|$, the global normal 
can be chosen to point in the direction of $X$. It also yields a strictly convex
component $U$ of $\{K>0\}$.  If this component does not coincide with $M$, the polar 
immersion $Y=\nu/K$ is complete and locally strictly convex. Van Heijenoort's 
theorem then identifies its image with the boundary of an unbounded convex set.  
Away from the recession directions, the support function of the polar body 
gives the local surface area formula $dS_C=h_C^{-1}d\omega$. The Finite area lemma
forces the recession set to have spherical measure zero, which leads to a 
global measure identity $h_C\,dS_C=d\omega$.  The interiority lemma then makes $h_C$
positive everywhere, contradicting its vanishing on the nonempty recession
set.  Thus the component is all of $M$, and the classification of
Brendle--Choi--Daskalopoulos completes the proof.

\medskip
The paper is organized as follows.  Section~\ref{sec:preliminaries} fixes
the geometric conventions, and states the results used later.  In
Section~\ref{sec:convex-component-polar-body}, we select the strictly convex component 
determined by a farthest point, and construct its polar immersion.  We then 
study the resulting unbounded convex set and the
support function of its polar body.  In Section~\ref{sec:global-measure-identity}, a
finite area lemma is proved, which will be used to pass from the local surface area formula to a
global measure identity. Finally, Section~\ref{sec:proof-main-theorem} completes the 
proof of Theorem~\ref{thm:main}.

\section{Preliminaries}\label{sec:preliminaries}

We begin by fixing some notation.
Let $m\geq2$, and $X:M^m\to\R^{m+1}$ be a smooth immersion of a
connected, closed, orientable manifold. Let $\nu$ be an arbitrary smooth global 
unit normal along $X$. The Euclidean metric and its Levi-Civita connection are
$g_{\mathrm{Euc}}:=\inner{\cdot}{\cdot}$ and
$\nabla^{\R^{m+1}}$, respectively. The induced metric, connection, and
volume element on $M$ are $g:=X^*g_{\mathrm{Euc}}$, $\nabla$, and $d\mu$.
On $\Sph^m$, the corresponding objects are $\sigma$, $\bar\nabla$, and
$d\omega$. With the convention $A=d\nu$ for the shape operator, the
Gauss formula for tangent vector fields $v,w$ is given by
\begin{equation}\label{eq:gauss-formula}
    \nabla^{\R^{m+1}}_vw
      =\nabla_vw-\inner{A(v)}{w}\nu.
\end{equation}

We next recall some basic concepts from convex geometry.

\begin{definition}\label{def:convex-body}
A \emph{convex body} in $\R^{m+1}$ is a compact convex set with nonempty
interior.
\end{definition}

\begin{definition}\label{def:polar-set}
Let $E\subset\R^{m+1}$ be a closed convex set containing the origin. The
\emph{polar set} of $E$ is
\begin{equation}\label{eq:polar-definition}
    E^\circ
      =\{x\in\R^{m+1}:\inner{x}{z}\leq1
          \text{ for every }z\in E\}.
\end{equation}
When $E^\circ$ is a convex body, we call it the \emph{polar body} of $E$.
\end{definition}

\begin{definition}\label{def:recession}
For a nonempty closed convex set $E\subset\R^{m+1}$, its \emph{recession cone} is
\begin{equation}\label{eq:rec-definition}
    \rec(E)
      =\{v\in\R^{m+1}:z+tv\in E
          \text{ for all }z\in E\text{ and }t\geq0\}.
\end{equation}
A convex cone $\cR$ is \emph{pointed} if
$\cR\cap(-\cR)=\{0\}$.
\end{definition}

If $C$ is a compact convex set, its \emph{support function} is
\begin{equation}\label{eq:support-definition}
    h_C(n)=\max_{x\in C}\inner{x}{n},
    \qquad n\in\Sph^m.
\end{equation}

We write $\cH^m$ for the $m$-dimensional Hausdorff measure and
$\dim_{\cH}E$ for the Hausdorff dimension of a set $E$.
We define the surface area measure of a convex body as follows.

\begin{definition}\label{def:surface-area-measure}
Let $C\subset\R^{m+1}$ be a convex body. The \emph{surface area measure} 
of $C$ is a finite Borel measure $S_C$ on $\Sph^m$ defined by
\[
    S_C(B)
      =\cH^m\bigl(\{x\in\partial C: \text{there exists } v\in B \text{ such that } \inner{x}{v}=h_C(v)\}\bigr)
\]
for every Borel set $B\subset\Sph^m$.
\end{definition}

For a symmetric $(0,2)$-tensor $T$ on $\Sph^m$, we write
\begin{equation}\label{eq:det_sigma}
    \det_\sigma T=\det(\sigma^{-1}T).
\end{equation}

If $h_C\in C^2(V)$ on an open set $V\subset\Sph^m$, define
\begin{equation}\label{eq:radius-tensor}
    r[h_C]:=\bar\nabla^2h_C+h_C\sigma
\end{equation}
to be the curvature radius tensor. If $r[h_C]$ is positive definite on $V$, 
then the corresponding part of the boundary is strictly convex. We also have the surface area formula
\begin{equation}\label{eq:area-density}
    dS_C=\det_\sigma(r[h_C])\,d\omega
    \qquad\text{on }V.
\end{equation}
These facts hold locally and do not require $h_C$ to be smooth on all of
$\Sph^m$, see \cite[Chapters~1 and~4]{Schneider2014} and
\cite[Remark~2.1]{ColesantiHugSaorin2014}.

We now list two theorems and the interiority lemma used in our paper.

\begin{theorem}[Van Heijenoort \cite{vanHeijenoort1952};
Jonker--Norman \cite{JonkerNorman1973}]
\label{thm:van-heijenoort}
Let $m\geq2$ and let $F:N^m\to\R^{m+1}$ be a complete locally convex 
immersion of a connected manifold without boundary. If $F$ is strictly 
locally convex at one point, then $F$ is an embedding onto the entire 
boundary of a closed convex set with nonempty interior.
\end{theorem}

The following lemma is the $p=0$ case of an
observation of Saroglou~\cite[Remark~3.5]{Saroglou2021}.  We give a
proof here for the reader's convenience. 

\begin{lemma}[Interiority lemma, Saroglou \cite{Saroglou2021}]
\label{lem:interiority}
Let $N\geq2$, $C\subset\R^N$ be a convex body with $0\in C$.
Let $h_C$ be the support function of $C$, $S_C$ denote its surface area measure.  If
\begin{equation}\label{eq:interiority-measure}
    h_C\,dS_C=d\omega
\end{equation}
as Borel measures on $\Sph^{N-1}$, then $0\in\Int C$.
\end{lemma}

\begin{proof}
Suppose, to the contrary, that $0\in\partial C$, and set
\begin{equation}\label{eq:interiority-T-D}
    T=\{v\in\Sph^{N-1}:h_C(v)=0\},
    \qquad
    D=\Sph^{N-1}\setminus T.
\end{equation}
Then $T$ is nonempty and compact.
Choose a closed ball
$\overline B_r(a)\subset C$ with $r>0$.  For every $v\in T$,
\[
    0=h_C(v)\geq\inner{a}{v}+r,
\]
and hence $\inner{a}{v}\leq-r$.  It follows that
$0\notin\operatorname{conv}(T)$.

Let $w$ be the point of minimum Euclidean norm in the compact convex set
$\operatorname{conv}(T)$.  Thus $w\neq0$.  For each $v\in T$, minimality
along the segment from $w$ to $v$ gives
\[
    0\leq
    \left.\frac{d}{dt}\right|_{t=0+}
       |w+t(v-w)|^2
    =2\inner{w}{v-w}.
\]
Extend $h_C$ to $\R^N$ as the positively homogeneous support function of
$C$. Convexity and nonnegativity of $h_C$ imply $h_C(w)=0$.
Consequently, the unit vector
$u=w/|w|$ satisfies
\begin{equation}\label{eq:u-separates-T}
    h_C(u)=0,
    \qquad
    \inner{u}{v}\geq |w|>0
    \quad\text{for every }v\in T.
\end{equation}

Applying \eqref{eq:interiority-measure} to $T$ gives $\omega(T)=0$.
Moreover, since $h_C>0$ on $D$, the same measure identity implies
\[
    dS_C=\frac{1}{h_C}\,d\omega
    \qquad\text{on }D.
\]
In particular,
\begin{equation}\label{eq:interiority-integrability}
    \int_D\frac{1}{h_C}\,d\omega=S_C(D)<\infty.
\end{equation}

Define $f=h_C^{-1}$ on $D$ and set $f=0$ on $T$.  Let
\[
    R_uv=v-2\inner{u}{v}u
\]
be reflection across $u^\perp$, and put
\[
    H_+=\{v\in\Sph^{N-1}:\inner{u}{v}>0\}.
\]
For $v\in H_+$, subadditivity of the support function gives
\[
    h_C(v)
      \leq h_C(R_uv)+2\inner{u}{v}h_C(u)
      =h_C(R_uv).
\]
Outside the set $T\cup R_uT$, this yields
$f(v)\geq f(R_uv)$.  Moreover, $h_C(-u)>0$. Otherwise,
$h_C(u)=h_C(-u)=0$ implies $C\subset u^\perp$, contrary to the fact that
$C$ has nonempty interior.  Since $R_uu=-u$, continuity gives
$h_C(v)<h_C(R_uv)$ on a spherical neighborhood of $u$ contained
in $H_+$.  Removing $T\cup R_uT$ from this neighborhood leaves a set
of positive spherical measure, and on this set $f(v)>f(R_uv)$.

Pairing the two hemispheres, and using \eqref{eq:interiority-integrability},
we obtain
\begin{equation}\label{eq:positive-first-moment}
\begin{aligned}
    I&:=\int_{\Sph^{N-1}}\inner{u}{v}f(v)\,d\omega(v)\\
     &=\int_{H_+}\inner{u}{v}
          \bigl(f(v)-f(R_uv)\bigr)\,d\omega(v)>0.
\end{aligned}
\end{equation}

The surface area measure of every convex body has zero first moment. 
Splitting this identity over $D$ and $T$, and using
\eqref{eq:u-separates-T}--\eqref{eq:positive-first-moment} gives
\[
\begin{aligned}
    0=\int_{\Sph^{N-1}}\inner{u}{v}\,dS_C(v)=I+\int_T\inner{u}{v}\,dS_C(v)>0,
\end{aligned}
\]
a contradiction.  Hence $0\in\Int C$.
\end{proof}

Once strict convexity has been established globally, the final rigidity 
result is provided by the following theorem.

\begin{theorem}[Brendle--Choi--Daskalopoulos
\cite{BrendleChoiDaskalopoulos2017}]\label{thm:BCD}
Let $m\geq2$.  A closed smooth strictly convex hypersurface in
$\R^{m+1}$ satisfying
\[
    K=\inner{X}{\nu}
\]
with respect to its outward normal is a round sphere.
\end{theorem}

\section{The convex component and the associated polar body}
\label{sec:convex-component-polar-body}

In this section, we derive differential identities from the shrinker equation
and use them to find a strictly convex component containing a point where
$|X|$ attains its maximum.  If this component does not coincide with $M$, we then
construct its polar immersion, and study the polar body of the convex set it bounds.

We write
\begin{equation}\label{eq:X-decomposition}
    X=X^\top+K\nu,
\end{equation}
where $X^\top$ is the tangential component of $X$.

\begin{lemma}\label{lem:differential-identities}
For any tangent vectors $v$, $w$,
\begin{align}
    \nabla K&=A(X^\top),\label{eq:grad-K}\\
    \nabla_vX^\top&=v-KA(v).\label{eq:grad-XT}
\end{align}
If $\rho:=\frac12|X|^2$, then
\begin{align}
    \nabla\rho&=X^\top,\label{eq:grad-rho}\\
    \nabla^2\rho(v,w)
      &=\inner{v}{w}-K\inner{A(v)}{w}.\label{eq:hess-rho}
\end{align}
\end{lemma}

\begin{proof}
Differentiating \eqref{eq:shrinker} gives
\[
    dK(v)=\inner{X}{d\nu(v)}
          =\inner{A(X^\top)}{v},
\]
which proves \eqref{eq:grad-K}.  Differentiating
\eqref{eq:X-decomposition} and taking tangential parts gives
\eqref{eq:grad-XT}.  Formula \eqref{eq:grad-rho} follows directly from
$d\rho(v)=\inner{X}{v}$.  Differentiating it and using
\eqref{eq:grad-XT} proves \eqref{eq:hess-rho}.
\end{proof}

We use these identities to determine the correct orientation at a farthest point
and single out a component on which the shape operator is positive definite.

\begin{lemma}
\label{lem:max-radius}
Choose $p_0\in M$ such that
\[
    R_0:=|X(p_0)|=\max_M|X|>0.
\]
The global unit normal may be chosen, without changing the shrinker
equation, so that
\begin{equation}\label{eq:normal-at-p0}
    \nu(p_0)=\frac{X(p_0)}{|X(p_0)|}.
\end{equation}
Let $U$ be the connected component of $\{K>0\}$
containing $p_0$.  Then
\begin{equation}\label{eq:positive-at-p0}
    K(p_0)=R_0>0,
    \qquad
    A_{p_0}\geq R_0^{-1}\id>0.
\end{equation}
Moreover, $A>0$ on $U$.
\end{lemma}

\begin{proof}
Since $p_0$ maximizes $\rho=\frac{1}{2}|X|^2$, one has
$X^\top(p_0)=0$.  Thus
\[
    X(p_0)=\lambda\nu(p_0),
    \qquad \lambda\in\{R_0,-R_0\}.
\]
The shrinker equation \eqref{eq:shrinker} gives $K(p_0)=\lambda$, while
\eqref{eq:hess-rho} and $\nabla^2\rho(p_0)\leq0$ give
\begin{equation}\label{eq:lambda-A}
    \lambda\inner{A_{p_0}v}{v}\geq |v|^2,
    \qquad v\in T_{p_0}M.
\end{equation}

Suppose first that $m$ is even.  If $\lambda<0$, then
\eqref{eq:lambda-A} implies $A_{p_0}<0$.  Since $m$ is even, this gives
$\det A_{p_0}>0$, contradicting
$\det A_{p_0}=K(p_0)=\lambda<0$.  Hence $\lambda=R_0$, so the normal
satisfying the shrinker equation already obeys \eqref{eq:normal-at-p0}.

If $m$ is odd and $\lambda<0$, replace $\nu$ by $-\nu$.  Under this
replacement
\[
    A\mapsto-A,
    \qquad
    K\mapsto(-1)^mK=-K,
    \qquad
    \inner{X}{\nu}\mapsto-\inner{X}{\nu},
\]
so the shrinker equation is preserved.  Thus in both cases we may assume
$X(p_0)=R_0\nu(p_0)$. The shrinker equation gives $K(p_0)=R_0$, and
\eqref{eq:lambda-A} yields
\[
    \inner{A_{p_0}v}{v}\geq R_0^{-1}|v|^2,
\]
which proves \eqref{eq:positive-at-p0}.  At every point of $U$,
\[
    \det A=K>0,
\]
so $A$ is a nondegenerate self-adjoint endomorphism.  By continuity of
the eigenvalues we have $A>0$ on $U$ as claimed.

\end{proof}

From now on, we use the normal and the component $U$
selected in Lemma~\ref{lem:max-radius}, and we retain the notation $\nu$,
$A$, and $K$. For the remainder of this paper, we assume for contradiction that
\begin{equation}\label{eq:U-neq-M}
    U\neq M.
\end{equation}
To study the geometry of $U$, we introduce the polar map
\begin{equation}\label{eq:Y-definition}
    Y:=\frac{\nu}{K}:U\longrightarrow\R^{m+1}.
\end{equation}
The next two lemmas show that $Y$ is a complete, locally strictly convex immersion.

\begin{lemma}\label{lem:Y-geometry}
Let $p\in U$ and $v,w\in T_pM$ be arbitrary. The map $Y$ is an immersion and
\begin{align}
    dY(v)&=\frac1K A(v)-\frac{dK(v)}{K^2}\nu,
       \label{eq:dY}\\
    \inner{X}{Y}&=1,\label{eq:XY-one}\\
    \inner{X}{dY(v)}&=0.\label{eq:X-dY-zero}
\end{align}
With the unit normal
\begin{equation}\label{eq:N-star}
    N^*:=\frac{X}{|X|},
\end{equation}
the second fundamental form of $Y$ is positive definite:
\begin{equation}\label{eq:Y-second-form}
    \inner{dY(v)}{dN^*(w)}
       =\frac{\inner{A(v)}{w}}{|X|K}.
\end{equation}
Thus $Y$ is locally strictly convex.
\end{lemma}

\begin{proof}
Formula \eqref{eq:dY} follows by direct differentiation.  Its component tangent
to the original immersion is $K^{-1}A(v)$, which vanishes only for $v=0$.
Hence $Y$ is an immersion.
The shrinker equation \eqref{eq:shrinker} gives \eqref{eq:XY-one}.  Differentiating it
and observing that $\inner{v}{Y}=0$ gives \eqref{eq:X-dY-zero}.
Since $K>0$ on $U$, $K=\inner{X}{\nu}$ implies $X\neq0$, so
$N^*$ is well defined.

For $w\in T_pM$,
\[
    dN^*(w)=\frac{w}{|X|}
      -\frac{X\inner{X}{w}}{|X|^3}.
\]
The second term is orthogonal to $dY(v)$ by
\eqref{eq:X-dY-zero}, while \eqref{eq:dY} and $w\perp\nu$ give
\[
    \inner{dY(v)}{w}=\frac1K\inner{A(v)}{w}.
\]
This proves \eqref{eq:Y-second-form}, which is positive definite.
\end{proof}

\begin{lemma}\label{lem:Y-proper}
The immersion $Y:U\to\R^{m+1}$ is proper and complete in its induced
metric.
\end{lemma}

\begin{proof}
Since $|\nu|=1$ and $K>0$,
\begin{equation}\label{eq:Y-norm}
    |Y|=\frac1K.
\end{equation}
Let $Q_0\subset\R^{m+1}$ be compact and choose $R>0$ such that
$Q_0\subset\overline B_R(0)$.  Then
\[
    Y^{-1}(Q_0)\subset U\cap\{K\geq R^{-1}\}.
\]
The set on the right hand side is compact in $U$. Since $Y^{-1}(Q_0)$ is closed, it is
compact.  Thus $Y$ is proper.

If $\{p_j\}$ is Cauchy for the metric induced by $Y$, then
\[
    |Y(p_i)-Y(p_j)|
       \leq d_{Y^*g_{\mathrm{Euc}}}(p_i,p_j).
\]
Consequently $\{Y(p_j)\}$ lies in a compact Euclidean ball.  Properness gives a
convergent subsequence of $\{p_j\}$ in $U$, and the Cauchy property then gives
convergence of the full sequence.  Hence the induced metric is complete.
\end{proof}

The preceding lemmas show that $Y$ is complete and locally strictly convex.
Van Heijenoort's theorem then identifies $Y(U)$ with the boundary of a
closed convex set.

\begin{proposition}\label{prop:C-star}
There is a closed convex set $C^*\subset\R^{m+1}$ with nonempty interior
such that
\[
    Y:U\longrightarrow\partial C^*
\]
is a diffeomorphism.  The vector $N^*=X/|X|$ is the outward unit normal, and
\begin{equation}\label{eq:C-star-halfspaces}
    C^*=\bigcap_{p\in U}
      \{z\in\R^{m+1}:\inner{z}{X(p)}\leq1\}.
\end{equation}
If $R_0=\max_M|X|>0$, then
\begin{equation}\label{eq:ball-in-C-star}
    \overline B_{1/R_0}(0)\subset C^*.
\end{equation}
Moreover, $C^*$ is unbounded.
\end{proposition}

\begin{proof}
Lemmas~\ref{lem:Y-geometry} and~\ref{lem:Y-proper} allow us to apply
Theorem~\ref{thm:van-heijenoort}.  It gives a closed convex set $C^*$ with
nonempty interior whose entire boundary is $Y(U)$.  Since $Y$ is a smooth
embedding and a local diffeomorphism onto its image, it is a diffeomorphism
onto $\partial C^*$.

At $Y(p)$, equations \eqref{eq:XY-one} and \eqref{eq:X-dY-zero} show that
the tangent hyperplane is
\[
    \{z\in\R^{m+1}:\inner{z}{X(p)}=1\}.
\]
The positivity in \eqref{eq:Y-second-form} and the
outward normal $N^*$ place the convex side in
$\{z\in\R^{m+1}:\inner{z}{X(p)}\leq1\}$, which proves
\eqref{eq:C-star-halfspaces}.

If $|z|\leq1/R_0$, then for every $p\in U$,
\[
    \inner{z}{X(p)}\leq |z|\,|X(p)|\leq1.
\]
This proves \eqref{eq:ball-in-C-star}. In particular,
$0\in\Int C^*$.

Finally, $U$ is noncompact.  Otherwise it would be a compact, hence 
closed subset of $M$.  Since $U$ is also open and $M$ is connected, 
this implies $U=M$, contrary to our assumption.
Therefore $\partial C^*=Y(U)$ is noncompact, so $C^*$ is unbounded.
\end{proof}

We next consider the recession cone of $C^*$. Since the convex set $C^*$ is 
unbounded, its recession cone describes the directions in which it extends 
to infinity. Set
\begin{equation}\label{eq:R-Gamma}
    \cR:=\rec(C^*),
    \qquad
    \Gamma:=\cR\cap\Sph^m.
\end{equation}

\begin{lemma}\label{lem:recession-facts}
The set $\cR$ is a closed convex cone, satisfying:
\begin{enumerate}[label=\textup{(\roman*)}]
\item if $z_0\in C^*$, then
\begin{equation}\label{eq:one-ray-test}
    z_0+tv\in C^*\ \text{for every }t\geq0
    \quad\Longleftrightarrow\quad
    v\in\cR;
\end{equation}
\item the cone $\cR$ contains a nonzero vector, and hence
$\Gamma\neq\varnothing$;
\item the cone $\cR$ is pointed.
\end{enumerate}
\end{lemma}

\begin{proof}
The general results of Rockafellar~\cite[Theorems~8.1--8.4]{Rockafellar1970}
show that $\cR$ is a closed convex cone and establish part~\textup{(i)}.
Since $C^*$ is unbounded, they also imply that $\cR$ contains a nonzero
vector, and hence $\Gamma\neq\varnothing$. This proves part~\textup{(ii)}.

Finally, suppose that $0\neq v\in\cR\cap(-\cR)$.  Then
\[
    C^*+tv=C^*
    \qquad\text{for every }t\in\R.
\]
So every boundary point lies on a complete boundary line in direction $v$, 
contradicting the positive definiteness of the second fundamental form in
\eqref{eq:Y-second-form}.  Therefore $\cR$ is pointed.
\end{proof}

Since $C^*$ contains the origin in its interior, radial projection is defined
on $\partial C^*$. The next proposition shows that it is a diffeomorphism onto
$\Sph^m\setminus\Gamma$.

\begin{proposition}
\label{prop:radial}
The radial projection
\begin{equation}\label{eq:radial-map}
    \pi:\partial C^*\longrightarrow\Sph^m\setminus\Gamma,
    \qquad
    \pi(y):=\frac{y}{|y|},
\end{equation}
is a diffeomorphism.  Consequently,
\begin{equation}\label{eq:nu-diffeomorphism}
    \nu:U\longrightarrow D:=\Sph^m\setminus\Gamma
\end{equation}
is a diffeomorphism.
\end{proposition}

\begin{proof}
Choose $r>0$ such that $B_r(0)\subset C^*$, and define
\[
    \rho_*(n):=\sup\{s\geq0:sn\in C^*\}\in(0,\infty],
    \qquad n\in\Sph^m.
\]
By Lemma~\ref{lem:recession-facts}(i), $\rho_*(n)=\infty$ exactly when
$n\in\Gamma$.  If $n\notin\Gamma$, then $\rho_*(n)n$ is the unique
intersection of the ray $\{sn:s\geq0\}$ with $\partial C^*$.  Indeed, if
$0\leq s<\rho_*(n)$, choose $s'$ with $s<s'<\rho_*(n)$.  The convex hull
of $B_r(0)$ and $s'n$ contains a neighborhood of $sn$, so $sn$ is an
interior point.  Points with $s>\rho_*(n)$ are outside $C^*$.

If $n\in\Gamma$, Lemma~\ref{lem:recession-facts}(i) gives
\[
    B_r(tn)=B_r(0)+tn\subset C^*,
    \qquad\text{for }t\geq0.
\]
Thus the entire ray is contained in the interior and does not meet the
boundary.  Hence $\pi$ is a bijection from $\partial C^*$ to
$\Sph^m\setminus\Gamma$.

For $y\neq0$, put $\widehat y=y/|y|$.  Then
\begin{equation}\label{eq:d-radial}
    d\pi_y(\xi)
      =\frac1{|y|}\left(
          \xi-\inner{\xi}{\widehat y}\widehat y
        \right).
\end{equation}
Thus the kernel of $d\pi_y$ is the line $\R y$.  If $y=Y(p)$,
then
\[
    T_y\partial C^*
       =\{\xi\in\R^{m+1}:\inner{\xi}{X(p)}=0\},
    \qquad
    \inner{y}{X(p)}=1.
\]
No nonzero vector parallel to $y$ is tangent to $\partial C^*$.
Therefore $\pi$ is a local diffeomorphism, thus a diffeomorphism.

Finally, since $K>0$ on $U$,
\[
    \pi(Y(p))
      =\frac{\nu(p)/K(p)}{|\nu(p)|/K(p)}
      =\nu(p).
\]
Together with the fact that $Y$ is a diffeomorphism, this proves
\eqref{eq:nu-diffeomorphism}.
\end{proof}

Thus the Gauss map parametrizes $U$ by the complement of the recession
directions.  We denote the polar set of $C^*$ by $C$ and study its 
support function.

\begin{lemma}\label{lem:C-body}
The set $C$ is a polar body containing the origin, and
\begin{equation}\label{eq:C-bounded}
    C\subset\overline B_{R_0}(0).
\end{equation}
\end{lemma}

\begin{proof}
The definition makes $C$ closed and convex and gives $0\in C$.  If
$x\in C\setminus\{0\}$, then
\[
    z=\frac{x}{R_0|x|}\in\overline B_{1/R_0}(0)\subset C^*.
\]
The polar inequality yields
\[
    \frac{|x|}{R_0}=\inner{x}{z}\leq1,
\]
which proves \eqref{eq:C-bounded}.  Hence $C$ is compact.

It remains to prove that $C$ has nonempty interior.  Suppose otherwise.
Since $0\in C$, its affine hull is a proper subspace of $\R^{m+1}$.  
Choose $0\neq v$ orthogonal to this subspace.  The bipolar
theorem gives
\[
    C^*=C^\circ,
\]
because $C^*$ is closed, convex, and contains the origin.  If $z\in C^*$,
$t\in\R$, and $x\in C$, then
\[
    \inner{x}{z+tv}=\inner{x}{z}\leq1.
\]
Thus $z+tv\in C^\circ=C^*$ for every $t\in\R$, and therefore
$v,-v\in\cR$.  This contradicts pointedness from
Lemma~\ref{lem:recession-facts}(iii).  Hence $C$ has nonempty interior.
\end{proof}

Since the polar body $C$ contains the origin, its support function 
denoted by $h:=h_C$ must be nonnegative.  
The next proposition describes the value of $h$ on $D$ and $\Gamma$.

\begin{proposition}\label{prop:h-identification}
For $D=\Sph^m\setminus\Gamma$, write $p(n):=\nu^{-1}(n)$.  Then
\begin{align}
    h(n)&=K(p(n))>0, &&n\in D,\label{eq:h-on-D}\\
    h(n)&=0, &&n\in\Gamma.\label{eq:h-on-Gamma}
\end{align}
In particular,
\begin{equation}\label{eq:h-zero-set}
    h^{-1}(0)=\Gamma.
\end{equation}
\end{proposition}

\begin{proof}
Fix $n\in D$ and set $p=p(n)$.  Since
\[
    Y(p)=\frac{n}{K(p)}\in C^*,
\]
every $x\in C$ satisfies
\[
    \inner{x}{\frac{n}{K(p)}}\leq1.
\]
Therefore $h(n)\leq K(p)$.  Conversely,
\eqref{eq:C-star-halfspaces} gives
$\inner{X(p)}{z}\leq1$ for every $z\in C^*$, so $X(p)\in C$.  Hence
\[
    h(n)\geq\inner{X(p)}{n}
      =\inner{X(p)}{\nu(p)}=K(p).
\]
This proves \eqref{eq:h-on-D}.

If $n\in\Gamma$, then $tn\in C^*$ for every $t\geq0$.  Consequently,
for every $x\in C$ and $t>0$,
\[
    t\inner{x}{n}\leq1.
\]
Letting $t\to\infty$ gives $\inner{x}{n}\leq0$.  Since $0\in C$, its
support function is nonnegative, and therefore $h(n)=0$.
\end{proof}

By \eqref{eq:h-on-D} we know that $h$ is smooth on $D$. The inverse Gauss 
parametrization and the curvature radius tensor used here are standard.  
We give a brief proof since the identification $h=K\circ\nu^{-1}$ is 
specific to the shrinker equation.

\begin{proposition}\label{prop:local-support-identities}
Let $p=\nu^{-1}:D\to U$ and $\widetilde X:=X\circ p$.
For $n\in D$, we identify $T_{p(n)}M$ with $T_n\Sph^m$ via
$dX_{p(n)}$, and $\sigma A^{-1}$ denotes the bilinear form
\[
    (\sigma A^{-1})(\zeta,\eta)
      =\sigma(A_{p(n)}^{-1}\zeta,\eta),
      \qquad\text{for }\zeta,\eta\in T_n\Sph^m.
\]
Then
\begin{align}
    \widetilde X&=\bar\nabla h+hn,\label{eq:X-support}\\
    \bar\nabla^2h+h\sigma&=\sigma A^{-1}>0,
       \label{eq:radii-A-inverse}\\
    \det_\sigma(\bar\nabla^2h+h\sigma)&=\frac1h.
       \label{eq:curvature-radius-determinant}
\end{align}
Moreover,
\begin{equation}\label{eq:finite-area-identity}
    \int_D\frac1h\,d\omega
      =\Vol_g(U)<\infty.
\end{equation}
\end{proposition}

\begin{proof}
Fix $n\in D$ and write $q=p(n)$. Under the identification above,
for any $\zeta\in T_n\Sph^m$,
\[
    d\widetilde X_n(\zeta)=A_q^{-1}\zeta.
\]
Using $h=K\circ p$ and \eqref{eq:grad-K}, we have
\[
    dh_n(\zeta)
      =\inner{A_qX(q)^\top}{A_q^{-1}\zeta}
      =\inner{X(q)^\top}{\zeta}.
\]
Thus $\bar\nabla h(n)=X(q)^\top$. Since
$h(n)=K(q)=\inner{X(q)}{n}$, decomposing $X(q)$ into its
tangential and normal parts proves \eqref{eq:X-support}.

Differentiating \eqref{eq:X-support}, for any
$\zeta,\eta\in T_n\Sph^m$ we obtain
\[
    (\bar\nabla^2h+h\sigma)(\zeta,\eta)
      =\inner{d\widetilde X_n(\zeta)}{\eta}
      =\sigma(A_q^{-1}\zeta,\eta).
\]
This proves \eqref{eq:radii-A-inverse} and it is positive
definite because $A_q>0$. Taking determinants and using
$\det A_q=K(q)=h(n)$ yields
\[
    \det_\sigma(\bar\nabla^2h+h\sigma)
       =\det(A_q^{-1})
       =\frac1{h(n)},
\]
which is \eqref{eq:curvature-radius-determinant}.

Finally, the Jacobian of the diffeomorphism $\nu:U\to D$ is
$\det A=K$. Therefore
\[
    \int_D\frac1h\,d\omega
      =\int_U\frac{K}{h\circ\nu}\,d\mu
      =\Vol_g(U)<\infty,
\]
because $h\circ\nu=K$ and $U\subset M$.
\end{proof}

\section{The global measure identity on the polar body \texorpdfstring{$C$}{C}}
\label{sec:global-measure-identity}

We now extend the local surface area formula on $D$ to the global identity
$h\,dS_C=d\omega$ on $\Sph^m$, which will be used to rule out $U\neq M$.
It remains to show that the zero set $\Gamma$ of $h$ has zero spherical measure.
The following finite area lemma yields this from
\eqref{eq:finite-area-identity}, the Lipschitz continuity of $h$, and the fact
that $\Gamma$ is the spherical section of a closed pointed convex cone.

\begin{lemma}[Finite area lemma]\label{lem:finite-area}
Let $m\geq2$, and $Q\subset\R^{m+1}$ be a nontrivial closed pointed convex
cone. Put
\[
    \Lambda:=Q\cap\Sph^m.
\]
Suppose that $f:\Sph^m\to[0,\infty)$ is Lipschitz and
\[
    f^{-1}(0)=\Lambda.
\]
If
\begin{equation}\label{eq:finite-hypothesis}
    \int_{\Sph^m\setminus\Lambda}\frac1f\,d\omega<\infty,
\end{equation}
then
\begin{equation}\label{eq:Lambda-dimension-bound}
    \dim_{\cH}\Lambda\leq m-1.
\end{equation}
In particular,
\begin{equation}\label{eq:Lambda-null}
    \omega(\Lambda)=\cH^m(\Lambda)=0.
\end{equation}
\end{lemma}

\begin{proof}
Since $Q\neq\{0\}$, the set $\Lambda$ is nonempty.
Define the dual cone
\[
    Q^\vee
      :=\{e\in\R^{m+1}:\inner{e}{q}\geq0
          \text{ for all }q\in Q\}.
\]
Since $Q$ is closed and convex, the bipolar theorem gives
$(Q^\vee)^\vee=Q$.  Pointedness implies
$\Int Q^\vee\neq\varnothing$.  Indeed, if $Q^\vee$ had empty interior,
its linear span would be a proper subspace of $\R^{m+1}$.  A nonzero vector $v$
orthogonal to that span satisfies
$v,-v\in(Q^\vee)^\vee=Q$, contradicting pointedness.

Choose a unit vector $e\in\Int Q^\vee$.  Then
\begin{equation}\label{eq:e-positive}
    \inner{e}{q}>0,
    \qquad q\in Q\setminus\{0\}.
\end{equation}
If equality held for some nonzero $q$, a sufficiently small
perturbation of $e$ in the direction $-q$ would remain in $Q^\vee$ but
have negative scalar product with $q$.  By compactness of $\Lambda$,
\begin{equation}\label{eq:uniform-positivity-on-Lambda}
    \min_{n\in\Lambda}\inner{e}{n}>0.
\end{equation}
Thus $\Lambda$ is contained in the hemisphere
\[
    H_e:=\{n\in\Sph^m:\inner{e}{n}>0\}.
\]

Consider the gnomonic projection and its inverse
\begin{equation}\label{eq:gnomonic}
\begin{aligned}
    \Phi:H_e&\longrightarrow e^\perp,
      &\Phi(n)&=\frac{n}{\inner{e}{n}}-e,\\
    \Psi:e^\perp&\longrightarrow H_e,
      &\Psi(x)&=\frac{e+x}{\sqrt{1+|x|^2}}.
\end{aligned}
\end{equation}
Set
\begin{equation}\label{eq:E-section}
    E:=\Phi(\Lambda)
      =\bigl(Q\cap\{z:\inner{e}{z}=1\}\bigr)-e.
\end{equation}
The right hand side shows that $E$ is closed and convex.  It is
bounded, and hence compact.

For $x,v\in e^\perp$, direct differentiation gives
\begin{equation}\label{eq:dPsi}
    |d\Psi_x(v)|^2
      =\frac{|v|^2}{1+|x|^2}
        -\frac{\inner{x}{v}^2}{(1+|x|^2)^2}.
\end{equation}
Hence on every closed Euclidean ball $\overline B_R\subset e^\perp$,
\begin{equation}\label{eq:Psi-bilipschitz}
    \frac1{1+R^2}|v|
       \leq |d\Psi_x(v)|\leq |v|.
\end{equation}
The Jacobian of $\Psi$ is given by
\begin{equation}\label{eq:Psi-Jacobian}
    J_\Psi(x)=(1+|x|^2)^{-\frac{m+1}{2}},\qquad x\in e^\perp\simeq\R^m.
\end{equation}

We next compare the Hausdorff dimensions of $E$ and $\Lambda$.
Choose $R_1$ so large that $E\subset B_{R_1}(0)$.  Let
$x,y\in\overline B_{R_1}(0)$,
put $n=\Psi(x)$ and $n'=\Psi(y)$, and set
$a=\inner{e}{n}$ and $b=\inner{e}{n'}$.  Then
\[
    a,b\geq c_{R_1}:=(1+R_1^2)^{-\frac{1}{2}}.
\]
Since $x=\Phi(n)$ and $y=\Phi(n')$,
\begin{equation}\label{eq:Phi-distance}
\begin{aligned}
    |\Phi(n)-\Phi(n')|
      &=\left|\frac{n}{a}-\frac{n'}{b}\right|\\
      &\leq \frac{1}{a}|n-n'|
          +\frac{|a-b|}{ab}|n'|\\
      &\leq\bigl((c_{R_1})^{-1}+(c_{R_1})^{-2}\bigr)|n-n'|\\
      &\leq C_{R_1}d_{\Sph^m}(n,n').
\end{aligned}
\end{equation}
Conversely, applying $\Psi$ to the Euclidean segment from $x$ to $y$ and
using the upper bound in \eqref{eq:Psi-bilipschitz} gives
\begin{equation}\label{eq:Psi-distance}
    d_{\Sph^m}(\Psi(x),\Psi(y))\leq|x-y|.
\end{equation}
Thus
$\Psi|_{\overline B_{R_1}(0)}$, and hence $\Phi$ on the corresponding compact
set, is bi-Lipschitz.  In particular,
\begin{equation}\label{eq:dimension-preserved}
    \dim_{\cH}\Lambda=\dim_{\cH}E.
\end{equation}

Let $L>0$ be the Lipschitz constant of $f$.
For $x\in e^\perp\setminus E$ and $y\in E$, we also have \eqref{eq:Psi-distance}.
Since $f(\Psi(y))=0$, we obtain
\begin{equation}\label{eq:f-distance-E}
    f(\Psi(x))\leq L|x-y|.
\end{equation}

Suppose $E$ has nonempty interior in
$e^\perp\simeq\R^m$.  Choose orthonormal coordinates
$x=(y,t)\in\R^{m-1}\times\R$ such that the projection $P$ of $E$ onto
the $y$-space has nonempty interior.  Define
\[
    \varphi(y):=\max\{t:(y,t)\in E\},
    \qquad y\in\Int P.
\]
The function $\varphi$ is finite and concave, hence locally Lipschitz on
$\Int P$.  Choose an open Euclidean ball $B$ with
$\overline B\Subset\Int P$ such that $\varphi$ is Lipschitz on a
neighborhood of $\overline B$.  Define
\begin{equation}\label{eq:full-collar}
    T:B\times(0,\varepsilon)\longrightarrow\R^m,
    \qquad
    T(y,s):=(y,\varphi(y)+s).
\end{equation}
The image of $T$ is disjoint from $E$, and $T$ is injective and Lipschitz.
For almost every $y\in B$ and every $s\in(0,\varepsilon)$, the Jacobian 
matrix of $T$ is
\[
    DT=
    \begin{pmatrix}
       I_{m-1}&0\\
       D\varphi&1
    \end{pmatrix},
    \qquad
    |\det DT|=1.
\]
Since $(y,\varphi(y))\in E$, formula \eqref{eq:f-distance-E} gives
\begin{equation}\label{eq:f-collar-bound}
    f(\Psi(T(y,s)))\leq Ls.
\end{equation}
Since $T(B\times(0,\varepsilon))$ is disjoint from $E$, its image under
$\Psi$ lies in $\Sph^m\setminus\Lambda$.
Choose $R_2$ so that $T(B\times(0,\varepsilon))\subset B_{R_2}(0)$. By
\eqref{eq:Psi-Jacobian},
\[
    J_\Psi\geq c_{R_2}:=(1+R_2^2)^{-\frac{m+1}{2}}.
\]
Using \eqref{eq:f-collar-bound}, we obtain
\[
\begin{aligned}
    \int_{\Sph^m\setminus\Lambda}\frac1f\,d\omega
      &\geq
        \int_{T(B\times(0,\varepsilon))}
          \frac{J_\Psi(x)}{f(\Psi(x))}\,dx\\
      &=\int_B\int_0^\varepsilon
        \frac{J_\Psi(T(y,s))}{f(\Psi(T(y,s)))}
        |\det DT(y,s)|\,ds\,dy\\
      &\geq\frac{c_{R_2}}{L}|B|
        \int_0^\varepsilon\frac{ds}{s}
       =\infty,
\end{aligned}
\]
contrary to \eqref{eq:finite-hypothesis}. Thus $E$ has empty interior. Since $E$ is convex, it follows that
\[
    \dim_{\cH}\Lambda=\dim_{\cH}E\leq m-1.
\]
This proves \eqref{eq:Lambda-dimension-bound}, and \eqref{eq:Lambda-null} follows.
\end{proof} 

We apply the lemma to the support function $h$ of the polar body $C$.

\begin{proposition}\label{prop:Gamma-null}
Under the assumption $U\neq M$, the set $\Gamma$ satisfies
\begin{equation}\label{eq:Gamma-nonempty-null}
    \Gamma\neq\varnothing,
    \qquad
    \omega(\Gamma)=0.
\end{equation}
In fact, $\dim_{\cH}\Gamma\leq m-1$.
\end{proposition}

\begin{proof}
Lemma~\ref{lem:recession-facts}(ii) gives $\Gamma\neq\varnothing$.  Since
$C$ is compact, its support function is Lipschitz.
Proposition~\ref{prop:h-identification} gives $h^{-1}(0)=\Gamma$, and
Proposition~\ref{prop:local-support-identities} gives
\[
    \int_{\Sph^m\setminus\Gamma}\frac1h\,d\omega<\infty.
\]
Applying Lemma~\ref{lem:finite-area} with $Q=\cR$ and $f=h$ proves the
claim.
\end{proof}

Thus the set $\Gamma$ has zero spherical measure.
This allows us to extend the local surface area formula to the whole sphere.

\begin{proposition}\label{prop:polar-body-measure-identity}
If $U\neq M$, then the polar body $C$ satisfies
\begin{equation}\label{eq:polar-body-measure-identity}
    h\,dS_C=d\omega
\end{equation}
as Borel measures on $\Sph^m$.
\end{proposition}

\begin{proof}
On $D$, Proposition~\ref{prop:local-support-identities} gives
\[
    \det_\sigma(\bar\nabla^2h+h\sigma)=\frac1h.
\]
Since $h$ is smooth on $D$, the local surface area formula 
\eqref{eq:area-density} gives
\begin{equation}\label{eq:SC-on-D}
    S_C(B)=\int_B\frac1h\,d\omega
\end{equation}
for every Borel set $B\subset D$.

Let $B\subset\Sph^m$ be an arbitrary Borel set.  Splitting it into its
parts in $D$ and $\Gamma$, using \eqref{eq:SC-on-D} and
$h|_\Gamma=0$, we obtain
\[
\begin{aligned}
    \int_B h\,dS_C
      &=\int_{B\cap D}h\,dS_C
        +\int_{B\cap\Gamma}h\,dS_C\\
      &=\int_{B\cap D}d\omega.
\end{aligned}
\]
Proposition~\ref{prop:Gamma-null} gives
$\omega(\Gamma)=0$, so the last quantity equals $\omega(B)$. This
proves \eqref{eq:polar-body-measure-identity}. 
\end{proof}

\section{Proof of Theorem~\ref{thm:main}}\label{sec:proof-main-theorem}

For the normal chosen in Lemma~\ref{lem:max-radius}, 
we have $A>0$ on the component $U$. Suppose that $U\neq M$.
The polar body $C$ contains the origin, and
Proposition~\ref{prop:polar-body-measure-identity} gives
\[
    h\,dS_C=d\omega.
\]
Lemma~\ref{lem:interiority} therefore yields $0\in\Int C$.  Hence
$h>0$ on $\Sph^m$.  This contradicts
Proposition~\ref{prop:h-identification} and
Lemma~\ref{lem:recession-facts}\textup{(ii)}, which give a nonempty set
$\Gamma$ on which $h=0$.  Therefore $U=M$.  Hence, with respect to the
selected normal,
\begin{equation}\label{eq:global-convexity}
    A>0,
    \qquad
    K>0
    \qquad\text{on }M.
\end{equation}

Since $M$ is compact, the metric induced by $X$ is complete.  By
\eqref{eq:global-convexity} we know that $X$ is strictly convex everywhere.
Theorem~\ref{thm:van-heijenoort} now applied to the original immersion
$X$ shows that $X$ is an embedding onto the boundary of a convex body.
Since $A>0$ with our convention $A=d\nu$, the selected normal $\nu$ is
the outward normal. Hence Theorem~\ref{thm:BCD} shows that
$X(M)$ is a round sphere.

Write this sphere as $\partial B_R(c)$, where $B_R(c)$ is a Euclidean ball 
in $\R^{m+1}$ centred at $c$ with $R>0$. At the point $X=c+Rn$, with $n\in\Sph^m$, 
the shrinker equation \eqref{eq:shrinker} becomes
\begin{equation}\label{eq:sphere-normalization}
    R^{-m}=\inner{c+Rn}{n}
      =\inner{c}{n}+R,
    \qquad n\in\Sph^m.
\end{equation}
Applying this equality at $n$ and $-n$ gives $\inner{c}{n}=0$ for every $n$, 
so $c=0$. The equation then gives $R^{-m}=R$, hence $R=1$. Therefore $X(M)=\Sph^m(1)$.

\noindent\textbf{Acknowledgments.}
The author would like to thank his supervisor Professor Hui Ma for invaluable 
support and helpful discussions. The author was supported by the National 
Natural Science Foundation of China (Grant Nos.~12471048 and W2521103).

\end{document}